\documentclass{birkjour}
 \newtheorem{theorem}{Theorem}[section]
 
 \newtheorem{lemma}[theorem]{Lemma}
 
 \theoremstyle{definition}
 \newtheorem{definition}[theorem]{Definition}
 \theoremstyle{remark}
 \newtheorem{rem}[theorem]{Remark}
 \newtheorem*{example}{Example}
 \numberwithin{equation}{section}

\begin{document}

%
%
%
%
%
%
%
%
%

\title[A Note on Quasi bi-slant submanifolds of cosymplectic manifolds]
 {A Note on Quasi bi-slant submanifolds of cosymplectic manifolds}

\author[M. A. Akyol]{Mehmet Akif Akyol$^1$}
\address{$^1$Bingol University\\ Faculty of Arts and Sciences,\\ Department of Mathematics\\ 12000, Bing\"{o}l, Turkey}
\email{mehmetakifakyol@bingol.edu.tr}
\author[S. Beyendi]{Selahattin Beyendi$^2$}
\address{$^2$\.{I}nonu University\\ Faculty of Education \\ 44000, Malatya, Turkey}
\email{selahattin.beyendi@inonu.edu.tr}






\date{January 1, 2004}

\begin{abstract}
The main purpose of the present paper is to define and study the notion of quasi bi-slant submanifolds of almost contact metric manifolds. We mainly concerned with quasi bi-slant submanifolds of cosymplectic manifolds as a generalization of slant, semi-slant, hemi-slant, bi-slant and quasi hemi-slant submanifolds. First, we give non-trivial examples in order to demostrate the method presented in this paper is effective and investigate the geometry of distributions. Moreover, We study these types of submanifolds with parallel canonical structures.
\end{abstract}

\maketitle
\section{Introduction}

Study of submanifolds theory has shown an increasing development in image processing, computer design, economic modeling as well as in mathematical physics and in mechanics. In this manner, B-Y. Chen \cite{C2} initiated the notion of slant submanifold as a generalization of both holomorphic (invariant) and totally real submanifold (anti-invariant) of an almost Hermitian manifold.  Inspried by B-Y. Chen's paper, many geometers have studied this notion in the different kind of structures: (see \cite{S1}, \cite{S2}). Many consequent results on slant submanifolds are collected in his book \cite{C1}. 
After this notion, as a generalization of semi-slant submanifold which was defined by N. Papaghiuc \cite{P} (see also \cite{CCFF1}). A. Carriazo \cite{C} and \cite{C-1} introduced the notion of bi-slant submanifold under the name anti-slant submanifold. However, B. \c{S}ahin called these submanifolds hemi-slant submanifolds in \cite{S}. (See also \cite{DA} and \cite{DA1}, \cite{PVK}, \cite{TO}). 

Furthermore, the submanifolds of a cosymplectic manifold have been studied by many geometers: See \cite{GHA}, \cite{Khan}, \cite{Kim},  \cite{LLS}, \cite{Lotta}, \cite{Ludden},  \cite{UOK}.
Taking into account of the above studies, we are motivated to fill a gap in the literature by giving the notion of quasi bi-slant submanifolds in which the tangent bundle consist of one invariant and two slant distributions and the Reeb vector field. In this paper, as a generalization of slant, semi-slant, hemi-slant, bi-slant and quasi hemi-slant submanifolds, we introduce quasi bi-slant submanifolds and investigate the geometry of distributions in detail.

The paper is organized as follows: In section 2, we recall basic formulas and definitions for a cosymplectic
manifold and their submanifolds. In section 3, we introduce the notion of quasi bi-slant submanifolds, giving an non-tirivial example and obtain some basic results for the next sections. In section 4, we give some necessary and sufficient conditions for the geometry of distributions. Finally, we study these types of submanifolds with parallel canonical structures.

\section{Preliminaries}

In this section, we give the definition of cosymplectic manifold and some background on submanifolds theory.

A $(2m+1)$-dimensional $C^\infty$-manifold $M$ said to have an almost contact structure if there exist
on $M$ a tensor field $\varphi$ of type (1,1), a vector field  $\xi$ and 1-form  $\eta$ satisfying:
\begin{equation}\label{phikare}
\varphi^2=-I+\eta\otimes\xi, \ \ \varphi\xi=0,\ \ \eta o\varphi=0,\ \ \eta(\xi)=1. 
\end{equation}
There always exists a Riemannian metric $g$ on an almost contact manifold $M$ satisfying the following
conditions
\begin{equation}\label{gphixphiy}
g(\varphi X,\varphi Y)=g(X,Y)-\eta(X)\eta(Y),\ \ \ \eta(X)=g(X,\xi) 
\end{equation}
where $X,Y$ are vector fields on $M.$

An almost contact structure $(\varphi,\xi,\eta)$ is said to be normal if the almost complex structure
$J$ on the product manifold $M \times \mathbb{R}$ is given by
$$J(X,f\frac{d}{dt})=(\varphi X-f\xi,\eta(X)\frac{d}{dt}),$$
where $f$ is a $C^\infty$-function on $M\times \mathbb{R}$ has no torsion i.e., $J$ is integrable. The condition
for normality in terms of $\varphi,\xi$ and $\eta$ is $[\varphi,\varphi]+2d\eta\otimes\xi=0$ on $M,$ where
$[\varphi,\varphi]$ is the Nijenhuis tensor of $\varphi.$ Finally, the fundamental two-form $\Phi$ is
defined $\Phi(X,\varphi Y)=g(X,\varphi Y).$

An almost contact metric structure $(\varphi,\xi,\eta,g)$ is said to be cosymplectic, if it is normal and both
$\Phi$ and $\eta$ are closed (\cite{B}, \cite{B1}, \cite{Ludden}), and the structure equation of a cosymplectic manifold is given by
\begin{equation}\label{nablaxvarphiy}
(\nabla_{X}\varphi)Y=0 
\end{equation}
for any $X,Y$ tangent to $M,$ where $\nabla$ denotes the Riemannian connection of the metric $g$ on $M.$ Moreover,
for cosymplectic manifold
\begin{equation}\label{nablaxxi}
\nabla_X\xi=0.
\end{equation}
\begin{example}(\cite{OL})\label{exm}
 $\mathbb{R}^{2n+1}$ with Cartesian coordinates $(x_i,y_i,z)(i=1,...,n)$ and its usual contact form $$\eta=dz \ \ \textrm{and} \ \ \xi=\frac{\partial}{\partial z},$$
	here $\xi$ is the characteristic vector field and its Riemannian metric $g$ and tensor
	field $\varphi$ are given by
	\begin{equation*}
	g=\sum_{i=1}^{n}((dx_i)^2+(dy_i)^2)+(dz)^2,\ \ \ \ \varphi=\left(
	\begin{array}{ccc}
	0 & \delta_{ij} & 0 \\
	-\delta_{ij} & 0 & 0 \\
	0 & 0 & 0 \\
	\end{array}
	\right),\ \ i=1,...,n.
	\end{equation*}
	This gives a cosymplectic manifold on $\mathbb{R}^{2n+1}.$ The vector fields
	$e_i=\frac{\partial}{\partial y_i},$ $e_{n+i}=\frac{\partial}{\partial x_i},$ $\xi$ form a $\varphi$-basis for the cosymplectic structure.
	On the other hand, it can be shown that $\mathbb{R}^{2n+1}(\varphi,\xi,\eta,g)$ is a cosymplectic manifold.
\end{example}

Let $M$ be a Riemannian manifold isometrically immersed in $\bar{M}$ and induced Riemannian metric on $M$ is denoted by the same symbol $g$ throughout this paper. Let $\mathcal{A}$ and $h$ denote the shape operator and second fundamental form, respectively, of immersion of $M$ into $\bar{M}$. The Gauss and Weingarten formulas of $M$ into $\bar{M}$ are given by \cite{C2}
\begin{equation}\label{gauss}
\bar{\nabla}_XY=\nabla_XY+h(X,Y)
\end{equation}
and
\begin{equation}\label{weingarten}
\bar{\nabla}_XV=-\mathcal{A}_VX+\nabla^{\perp}_XV,
\end{equation}
for any vector fields $X, Y\in \Gamma(TM)$ and $V\in\Gamma(T^{\perp}M)$, where $\nabla$ is the induced connection on $M$ and $\nabla^{\perp}$ represents the connection on the normal bundle $T^{\perp}M$ of $M$ and $A_V$ is the shape operator of $M$ with respect to normal vector $V\in\Gamma(T^{\perp}M)$. Moreover, $\mathcal{A}_V$ and $h$ are related by
\begin{equation}\label{ghxyv}
g(h(X,Y),V)=g(\mathcal{A}_VX,Y)
\end{equation}
for any vector fields $X, Y \in\Gamma(TM)$ and $V\in\Gamma(T^{\perp}M)$.


If $h(X,Y)=0$ for all $X,Y\in\Gamma(TM)$, then $M$ is said to be totally geodesic.

\section{Quasi bi-slant submanifolds of cosmyplectic manifolds}

In this section, we define the concept of quasi bi-slant submanifolds of cosymplectic manifolds, giving a non-trivial exmaple and obtain some related results for later use. 

\begin{definition}\label{def1}
A submanifold $M$ of cosymplectic manifolds $(\bar{M},\varphi,\xi,\eta,\bar{g})$
is called quasi bi-slant if there exists four orthogonal distributions
$\mathcal{D},$ $\mathcal{D}_{1}$ and $\mathcal{D}_{2}$ of $M,$ at the point $p\in M$  such that
\begin{itemize}
	\item[(i)] $TM=\mathcal{D}\oplus \mathcal{D}_{1}\oplus \mathcal{D}_{2}\oplus<\xi>$
	\item[(ii)] The distribution $\mathcal{D}$ is invariant, i.e. $\varphi \mathcal{D}=\mathcal{D}.$  
	\item[(iii)] $\varphi \mathcal{D}_{1}\perp \mathcal{D}_{2}$ and $\varphi \mathcal{D}_{2}\perp \mathcal{D}_{1};$
	\item [(iv)] The distributions $\mathcal{D}_1, \mathcal{D}_2$ are slant with slant angle $\theta_1$, $\theta_2,$ respectively.
\end{itemize}
\end{definition}

Taking the dimension of distributions $\mathcal{D},$ $\mathcal{D}_1$ and $\mathcal{D}_2$ are  $m_1, m_2$ and $m_3$, respectively. One can easily see the following cases:

\begin{itemize}
	\item If $m_1\neq0$ and  $m_2=m_3=0,$ then $M$ is a invariant submanifold. 
	\item If $m_1=m_2=0$ and $\theta_2=\frac{\pi}{2}$ then $M$ is an anti-invariant submanifold. 
	\item If $m_1=0, m_2\neq m_3\neq0,$ $\theta_1=0$ and $\theta_2=\frac{\pi}{2}$ then $M$ is a semi-invariant submanifold.
	\item If $m_1=m_2=0$ and $0<\theta_2<\frac{\pi}{2}$ then $M$ is a slant submanifold. 
	\item If $m_1=0, m_2\neq m_3\neq0,$ $\theta_1=0$ and $0<\theta_2<\frac{\pi}{2}$ then $M$ is a semi-slant submanifold.
	\item If $m_1=0, m_2\neq m_3\neq0,$ $\theta_1=\frac{\pi}{2}$ and $0<\theta_2<\frac{\pi}{2}$ then $M$ is a hemi-slant submanifold.
	\item If $m_1=0, m_2\neq m_3\neq0,$ and $\theta_1$ and  $\theta_2$ are different from either $0$  and $\frac{\pi}{2},$ then $M$ is a bi-slant submanifold.\\ 
\end{itemize}



If $m_1\neq m_2 \neq m_3\neq0$ and $\theta_1, \theta_2\neq0, \frac{\pi}{2},$ then $M$ is called a \textit{\bf{proper quasi bi-slant submanifold}}. 

\begin{rem}
	In this paper, we assume that $M$ is \textit{proper} quasi bi-slant submanifold of a cosymplectic manifold $\bar{M}.$
\end{rem}

Now, we present an example of proper quasi bi-slant submanifold in $\mathbb{R}^{11}$.

\begin{example}
	We will use the canonical contact structure $\varphi$ defined by $$\varphi(x_1,y_1,...,x_n,y_n,z)=(y_1,-x_1,...,y_n,-x_n,0).$$ Thus we have $\varphi(\partial x_i)=\partial y_i,
	$ $\varphi(\partial y_j)=-\partial x_j$ and $\varphi (\partial z)=0,\ \ 1\leq i,j\leq 5$ where $\partial x_i=\frac{\partial}{\partial x_i}$. 
	For any pair of real numbers {$\theta_1,\theta_2$} satisfying $0<\theta_1,\theta_2<\frac{\pi}{2},$ let us consider submanifold $M_{\theta_1, \theta_2}$ of $\mathbb{R}^{11}$ defined by 
	
	$\pi_{\theta_1, \theta_2}(u,s,w,k,t,r,z)=(u,s\cos\theta_1,0,s\sin\theta_1, \omega, k\cos\theta_2,0,k\sin\theta_2,t,r,z).$ If we take $$e_1=\partial x_1, \ \ e_2=\cos\theta_1\partial y_1+\sin\theta_1\partial y_2,$$ $$e_3=\partial x_3,\ \ e_4=\cos\theta_2\partial y_3+\sin\theta_2\partial y_4,$$ $$e_5=\partial x_5, \ \ e_6=\partial y_5, \ \ e_7=\xi=\partial z$$ then the restriction of ${e_1,...,e_7}$ to $M$ forms an orthonormal frame of the tangent bundle $TM.$ Obviously, we get
	$$\varphi e_1=\partial y_1, \ \varphi e_2=-\cos\theta_1\partial x_1-\sin\theta_1\partial x_2,\ \varphi e_3=\partial y_3$$
	$$\varphi e_4=-\cos\theta_2\partial x_3-\sin\theta_2\partial x_4, \ \varphi e_5=\partial y_5, \ \varphi e_6=-\partial x_5.$$ Let us put $\mathcal{D}_1=Span\{e_1,e_2\},$ $\mathcal{D}_2=Span\{e_3,e_4\},$ and $\mathcal{D}=Span\{e_5,e_6\}.$ Then obviously $\mathcal{D}_1,$ $\mathcal{D}_2$ and $\mathcal{D},$ satisfy the definition of quasi bi-slant submanifold $M_{\theta_1, \theta_2}$ defined by $\pi_{\theta_1, \theta_2}$ is a proper quasi bi-slant submanifold of $\mathbb{R}^{11}$ with {$\theta_1,\theta_2$} as its bi-slant angles.
\end{example}

Let $M$ be a quasi bi-slant submanifold of a cosymplectic manifold $\bar{M}.$ Then, for any $X\in \Gamma(TM),$ we have
\begin{align}\label{x}
X=\mathcal{P}X+\mathcal{Q}X+\mathcal{R}X+\eta(X)\xi
\end{align}
where $\mathcal{P}, \mathcal{Q}$ and $\mathcal{R}$ denotes the projections on the distributions $\mathcal{D}, \mathcal{D}_1$ and $\mathcal{D}_2,$ recpectively. 
\begin{equation}\label{vx}
\varphi X=\mathcal{T}X+\mathcal{F}X,
\end{equation}
where $\mathcal{T}X$ and $\mathcal{F}X$ are tangential and normal components on $M.$ Making now use of \eqref{x} and \eqref{vx}, we get immediately
\begin{equation}\label{vx1}
\varphi X=\mathcal{T}\mathcal{P}X+\mathcal{T}\mathcal{Q}X+\mathcal{F}\mathcal{Q}X
+\mathcal{T}\mathcal{R}X+\mathcal{F}\mathcal{R}X,
\end{equation}
here since $\varphi\mathcal{D}=\mathcal{D},$ we have $\mathcal{F}\mathcal{P}X=0.$ Thus we get
\begin{equation}\label{vtm}
\varphi(TM)=\mathcal{D}\oplus\mathcal{T}\mathcal{D}_1\oplus\mathcal{T}\mathcal{D}_2
\end{equation}
and
\begin{equation}\label{tdikm}
T^\perp M=\mathcal{F}\mathcal{D}_1\oplus\mathcal{F}\mathcal{D}_2\oplus\mu
\end{equation}
where $\mu$ is the orthogonal complement of $\mathcal{F}\mathcal{D}_1\oplus\mathcal{F}\mathcal{D}_2$ in 
$T^\perp M$ and it is invariant with recpect to $\varphi.$
Also, for any $Z\in T^\perp M,$ we have
\begin{equation}\label{vz}
\varphi Z=\mathcal{B}Z+\mathcal{C}Z,
\end{equation}
where $\mathcal{B}Z\in \Gamma(TM)$ and $\mathcal{C}Z\in \Gamma(T^\perp M).$

Taking into account of  the condition (iii) in Definition \eqref{def1}, \eqref{vx} and \eqref{vz}, we obtain the followings:
\begin{lemma}
Let $M$ be a quasi bi-slant submanifold of a cosymplectic manifold $\bar{M}.$ Then, we have
\begin{equation*}
\textbf{(a)}\ T\mathcal{D}_{1}\subset \mathcal{D}_{1},\quad
\textbf{(b)}\ T\mathcal{D}_{2}\subset \mathcal{D}_{2},\quad
\textbf{(c)}\ \mathcal{B}\mathcal{F}\mathcal{D}_{1}=\mathcal{D}_{1},\quad
\textbf{(d)}\ \mathcal{B}\mathcal{F}\mathcal{D}_{2}=\mathcal{D}_{2}.
\end{equation*}
\end{lemma}
With the help of \eqref{vx} and \eqref{vz}, we obtain the following Lemma.

\begin{lemma}\label{general}
	Let $M$ be a quasi bi-slant submanifold of a cosymplectic manifold $\bar{M}.$ Then, we have
	\begin{equation*}
	\textbf{(a)}\, \mathcal{T}^{2}U_1=-\cos^{2}\theta_{1}U_1, \quad \textbf{(b)}\,  \mathcal{T}^{2}U_2=-\cos^{2}\theta_{2}U_2,
	\end{equation*}
	\begin{equation*}
	\textbf{(c)}\, \mathcal{B}\mathcal{F}U_1=-\sin ^{2}\theta_{1}U_1, \quad \textbf{(d)}\, \mathcal{B}\mathcal{F}U_2=-\sin ^{2}\theta_{1}U_2,
	\end{equation*}
	\begin{equation*}
	\textbf{(e)}\, \mathcal{T}^{2}U_1+\mathcal{B}\mathcal{F}U_1=-U_1, \quad \textbf{(f)}\, \mathcal{T}^{2}U_2+\mathcal{B}\mathcal{F}U_2=-U_2,
	\end{equation*}
	\begin{equation*}
	\textbf{(g)}\, \mathcal{F}\mathcal{T}U_1+\mathcal{C}\mathcal{F}U_1=0, \quad \textbf{(h)}\, \mathcal{F}\mathcal{T}U_2+\mathcal{C}\mathcal{F}U_2=0,
	\end{equation*}
	for any $U_1 \in \mathcal{D}_1$ and $U_2 \in \mathcal{D}_2$.
\end{lemma}

By using  \eqref{nablaxvarphiy}, Definition \eqref{def1}, \eqref{vx} and \eqref{vz}, we obtain the following Lemma.
\begin{lemma}\label{lemtkare}
	Let $M$ be a quasi bi-slant submanifold of a cosymplectic manifold $\bar{M}.$ Then, we have
\begin{enumerate}
	\item[(i)] $\mathcal{T}_{i}^{2}U_i=-\cos^{2}\theta_{i}U_i,$
	\item[(ii)] $g(\mathcal{T}_i{U_i}, \mathcal{T}_i{V_i})=(\cos^{2}\theta_{i})g(U_i,V_i),$
	\item[(iii)] $g(\mathcal{F}_i{U_i}, \mathcal{F}_i{V_i})=(\sin^{2}\theta_{i})g(U_i,V_i)$
\end{enumerate}
for any $i=1,2,$ $U_1,V_1\in\Gamma({\mathcal{D}_1})$ and $U_2,V_2\in\Gamma({\mathcal{D}_2}).$
\end{lemma}

We need the following lemma for later use.
\begin{lemma}
	Let $M$ be a quasi bi-slant submanifold of a cosymplectic manifold $\bar{M},$ then for any $Z_1,Z_2\in \Gamma(TM),$ we have the following
	\begin{align}\label{teget}
		\nabla_{Z_1}\mathcal{T}Z_2-\mathcal{T}\nabla_{Z_1}Z_2=\mathcal{A}_{\mathcal{F}Z_2}Z_1+\mathcal{B}
		h(Z_1,Z_2)
	\end{align}
	and
	\begin{align}\label{normal}
	\nabla^{\perp}_{Z_1}\mathcal{F}Z_2-\mathcal{F}\nabla_{Z_1}Z_2=\mathcal{C}
	h(Z_1,Z_2)-h(Z_1,\mathcal{T}Z_2).
	\end{align}
\end{lemma}

\begin{proof}
	Since $\bar{M}$ is a cosmyplectic manifold, we have that $$(\bar{\nabla}_{Z_1}\varphi)Z_2=0$$ which implies that $$\bar{\nabla}_{Z_1}\varphi Z_2-\varphi\bar{\nabla}_{Z_1}Z_2=0.$$ By using \eqref{gauss} and \eqref{vx}, we get
	$$\bar{\nabla}_{Z_1}\mathcal{T}Z_2+\bar{\nabla}_{Z_1}\mathcal{F}Z_2-\varphi(\nabla_{Z_1}Z_2+h(Z_1,Z_2))=0.$$ Taking into account of \eqref{gauss}, \eqref{weingarten}, \eqref{vx} and \eqref{vz}, we obtain
	\begin{align*}
	&\nabla_{Z_1}\mathcal{T}Z_2+h(Z_1,\mathcal{T}Z_2)-\mathcal{A}_{\mathcal{F}Z_2}Z_1+\nabla^{\perp}_{Z_1}\mathcal{F}Z_2\\
	&-\mathcal{T}\nabla_{Z_1}Z_2-\mathcal{F}\nabla_{Z_1}Z_2-\mathcal{B}h{Z_1,Z_2}-\mathcal{C}h(Z_1,Z_2)=0.
	\end{align*}
	Comparing the tangential and normal components, we have the required results.
\end{proof}

In a similar way, we have:

\begin{lemma}
	Let $M$ be a quasi bi-slant submanifold of a cosymplectic manifold $\bar{M},$ then  we have the following
	\begin{align}\label{teget1}
	\nabla_{Z_1}\mathcal{B}W_1-\mathcal{B}\nabla^{\perp}_{Z_1}W_1=\mathcal{A}_{\mathcal{C}W_1}Z_1
	-\mathcal{T}\mathcal{A}_{W_1}Z_1
	\end{align}
	and
	\begin{align}\label{normal1}
	\nabla^{\perp}_{Z_1}\mathcal{C}W_1-\mathcal{C}\nabla^{\perp}_{Z_1}W_1=-\mathcal{F}
	\mathcal{A}_{W_1}Z_1-h(Z_1,\mathcal{B}W_1)
	\end{align}
	for any $Z_1\in \Gamma(TM)$ and $W_1\in \Gamma(T^\perp M).$
\end{lemma}

\section{Integrability and Totally geodesic foliations}
In this section we give some necessary and sufficient condition for  the integrability of the distributions.

First, we have the following theorem:
\begin{theorem}
Let $M$ be a quasi bi-slant submanifolds of $\bar{M}.$ The invariant distribution $\mathcal{D}$ is integrable if and only if 
\begin{align*}
g(\mathcal{T}(\nabla_X\mathcal{T}Y-\nabla_Y\mathcal{T}X), Z)
&=g(h(X,\mathcal{T}Y)-h(Y,\mathcal{T}X),\varphi QZ+\varphi RZ)
\end{align*}
for any $X,Y\in \Gamma(\mathcal{D})$ and $Z\in\Gamma(\mathcal{D}_1 \oplus \mathcal{D}_2).$
\end{theorem}
\begin{proof}
	The distribution $\mathcal{D}$ is integrable on $M$ if and only if $$g([X,Y],\xi)=0\quad \mathrm{and}\quad g([X,Y],Z)=0$$ for any $X,Y\in \Gamma(\mathcal{D})$, $Z\in\Gamma(\mathcal{D}_1 \oplus \mathcal{D}_2)$ and $\xi \in \Gamma(TM).$ Since $M$ is a cosymplectic manifold, we immediately have $g([X,Y],\xi)=0.$ Thus $\mathcal{D}$ is integrable if and only if $g([X,Y],Z)=0.$ Now, for any $X,Y\in \mathcal{D}$ and $Z=\mathcal{Q}Z+\mathcal{R}Z\in \Gamma(\mathcal{D}_1\oplus \mathcal{D}_2),$ by using \eqref{gphixphiy}, \eqref{gauss}, we obtain
\begin{align*}
		g([X,Y],Z)&=g(\varphi\bar{\nabla}_XY,\varphi Z)-\eta({\bar{\nabla}_XY})\eta(Z)-g(\varphi\bar{\nabla}_YX,\varphi Z)+\eta({\bar{\nabla}_YX})\eta(Z).
\end{align*}		
Now, using \eqref{nablaxxi}, \eqref{vx} and $\mathcal{F}Y=0$ for any $Y\in\Gamma(\mathcal{D}),$ we have		
\begin{align*}
g([X,Y],Z)&=g(\bar{\nabla}_X\varphi Y,\varphi Z)-  g(\bar{\nabla}_Y\varphi X,\varphi Z)\\
&=g(\bar{\nabla}_{X}\mathcal{T}Y,\varphi Z)-  g(\bar{\nabla}_{Y}\mathcal{T}X,\varphi Z).
\end{align*}
Taking into account of \eqref{gauss} and \eqref{vx1} in the above equation, we get 
\begin{align*}
g([X,Y],Z)&=-g(\varphi\nabla_X\mathcal{T}Y, Z)+g(h(X,\mathcal{T}Y),\varphi Z)\\
&+g(\varphi\nabla_{Y}\mathcal{T}X,Z)-g(h(Y,\mathcal{T}X),\varphi Z).
\end{align*}
Now again taking into account the equation \eqref{vx}, we obtain		
\begin{align*}
g([X,Y],Z)&=g(\mathcal{T}(\nabla_Y\mathcal{T}X-\nabla_X\mathcal{T}Y), Z)\\
&+g(h(X,\mathcal{T}Y)-h(Y,\mathcal{T}X),\varphi QZ+\varphi RZ)
\end{align*}
which completes the proof.
\end{proof}

For the slant distrubition $\mathcal{D}_1,$ we have:
\begin{theorem}
	Let $M$ be a quasi bi-slant submanifolds of $\bar{M}.$ The slant distribution $\mathcal{D}_1$ is integrable if and only if 
	\begin{align*}
g(\nabla^{\perp}_{U_1}\mathcal{F}V_1+\nabla^{\perp}_{V_1}\mathcal{F}U_1,\mathcal{F}\mathcal{R}Z)&=	
g(\mathcal{A}_{\mathcal{F}\mathcal{T}V_1}U_1-\mathcal{A}_{\mathcal{F}\mathcal{T}U_1}V_1,Z)\\
&+g(\mathcal{A}_{\mathcal{F}V_1}U_1+\mathcal{A}_{\mathcal{F}U_1}V_1,\mathcal{T}Z)
\end{align*}
for any $U_1,V_1\in \Gamma(\mathcal{D}_1)$, $Z\in\Gamma(\mathcal{D}\oplus \mathcal{D}_2).$
\end{theorem}
\begin{proof}
	The distribution $\mathcal{D}_1$ is integrable on $M$ if and only if $$g([U_1,V_1],\xi)=0\quad \mathrm{and}\quad g([U_2,V_2],Z)=0$$ for any $U_1,V_1\in \Gamma(\mathcal{D}_1)$, $Z\in\Gamma(\mathcal{D}\oplus \mathcal{D}_2)$ and $\xi \in \Gamma(TM).$ The first case is trivial. Thus $\mathcal{D}_1$ is integrable if and only if $g([U_1,V_1],Z)=0.$ Now, for any $U_1,V_1\in \mathcal{D}_1$ and $Z=\mathcal{P}Z+\mathcal{R}Z\in \Gamma(\mathcal{D}\oplus \mathcal{D}_2),$ by using \eqref{gphixphiy}, \eqref{gauss}, we obtain
		\begin{align*}
			g([U_1,V_1],Z)&=
			-g(\bar{\nabla}_{U_1}\varphi \mathcal{T}V_1, Z)-g(\bar{\nabla}_{U_1}\mathcal{F}V_1,\varphi Z)\\
			&+g(\bar{\nabla}_{V_1}\varphi \mathcal{T}U_1, Z)-g(\bar{\nabla}_{V_1}\mathcal{F}U_1,\varphi Z)
	\end{align*}		
	Taking into account the equation lemma \eqref{lemtkare} (i) in the above equation, we get			
\begin{align*}
	g([U_1,V_1],Z)&=
			\cos^2\theta_1g([U_1,V_1],Z)-g(\bar{\nabla}_{U_1}\mathcal{F}\mathcal{T}V_1-\bar{\nabla}_{V_1}\mathcal{F}\mathcal{T}U_1,Z)\\
			&-g(\bar{\nabla}_{U_1}\mathcal{F}V_1+\bar{\nabla}_{V_1}\mathcal{F}U_1,\varphi \mathcal{P}Z+\varphi \mathcal{R}Z).
		\end{align*}
Now, using \eqref{weingarten} and \eqref{vx1}, we obtain
 	\begin{align*}
 		g([U_1,V_1],Z)&=
			\cos^2\theta_1g([U_1,V_1],Z)+g(\mathcal{A}_{\mathcal{F}\mathcal{T}V_1}U_1-\mathcal{A}_{\mathcal{F}\mathcal{T}U_1}V_1,Z)\\
			&+g(\mathcal{A}_{\mathcal{F}V_1}U_1+\mathcal{A}_{\mathcal{F}U_1}V_1,\mathcal{T}Z)\\
			&-g(\nabla^{\perp}_{U_1}\mathcal{F}V_1+\nabla^{\perp}_{V_1}\mathcal{F}U_1,\mathcal{F}\mathcal{R}Z)
	\end{align*}
or 
	\begin{align*}
	\sin^2\theta_1g([U_1,V_1],Z)&=g(\mathcal{A}_{\mathcal{F}\mathcal{T}V_1}U_1-\mathcal{A}_{\mathcal{F}\mathcal{T}U_1}V_1,Z)\\
	&+g(\mathcal{A}_{\mathcal{F}V_1}U_1+\mathcal{A}_{\mathcal{F}U_1}V_1,\mathcal{T}Z)\\
	&-g(\nabla^{\perp}_{U_1}\mathcal{F}V_1+\nabla^{\perp}_{V_1}\mathcal{F}U_1,\mathcal{F}\mathcal{R}Z)	
	\end{align*}
which gives the assertion.
\end{proof}
		In a similar way, we obtain the following case for the slant distribution $\mathcal{D}_2.$
\begin{theorem}
	Let $M$ be a quasi bi-slant submanifolds of $\bar{M}.$ The slant distribution $\mathcal{D}_2$ is integrable if and only if 
		\begin{align*}
	\mathcal{T}(\nabla_{U_2}\mathcal{T}V_2-\mathcal{A}_{\mathcal{F}V_2}U_2)\in\Gamma(\mathcal{D}_2),
	\end{align*}
	\begin{align*}
	\mathcal{B}(h(U_2,TV_2)+\nabla^{\perp}_{U_2}\mathcal{F}V_2)\in\Gamma(TM)^{\perp}
	\end{align*}
	and
	\begin{align*}
	g(\mathcal{A}_{\mathcal{F}Z}V_2-\nabla_{V_2}\mathcal{T}Z,\mathcal{T}U_2)=g(h(V_2,\mathcal{T}Z)+
	\nabla^{\perp}_{V_2}\mathcal{F}Z,\mathcal{F}U_2)
	\end{align*}
	for any $U_2, V_2\in\Gamma(\mathcal{D}_2),\ Z=\mathcal{P}Z+\mathcal{Q}Z\in\Gamma(\mathcal{D}\oplus \mathcal{D}_1)$ and $W\in\Gamma(TM)^\perp.$
\end{theorem}	

\begin{theorem}\label{dparalel}
	Let $M$ be a quasi bi-slant submanifolds of $\bar{M}.$ The invariant distribution $\mathcal{D}$ defines totally geodesic foliation on $M$ if and only if 
	\begin{align}\label{dpar1}
	g(\nabla_{X}\mathcal{T}Y,\mathcal{T}Z)=-g(h(X,\mathcal{T}Y),\mathcal{F}Z)
	\end{align}
	and
	\begin{align}\label{dpar2}
	\mathcal{F}\nabla_{X}\mathcal{T}Y_1+\mathcal{C}h(X,\mathcal{T}Y)\in\Gamma(TM)
	\end{align}
for any $X, Y\in\Gamma(\mathcal{D}),\ Z=\mathcal{Q}Z+\mathcal{R}Z\in\Gamma(\mathcal{D}_1\oplus \mathcal{D}_2)$ and $W\in\Gamma(TM)^\perp.$
\end{theorem}	

\begin{proof}
The distribution $\mathcal{D}$ defines a totaly geodesic foliation on $M$ if and only if $g(\bar{\nabla}_{X}Y,\xi)=0,$ $g(\bar{\nabla}_{X}Y,Z)=0$ and $g(\bar{\nabla}_{X}Y,W)=0$ for any $X,Y\in\Gamma(\mathcal{D}),$ $Z=\mathcal{Q}Z+\mathcal{R}Z\in\Gamma(\mathcal{D}_1\oplus \mathcal{D}_2)$ and $W\in\Gamma(TM)^\perp.$ Then by using 
\eqref{gphixphiy} and \eqref{nablaxxi}, we obtain
	\begin{equation}\label{dpar}
	g(\bar{\nabla}_{X}Y,\xi)=Xg(Y,\xi)-g(Y,\bar{\nabla}_{X}\xi)=-g(Y,\bar{\nabla}_{X}\xi)=0.
	\end{equation}
On the other hand, using \eqref{gphixphiy}, we find
	\begin{align*}
		g(\bar{\nabla}_{X}Y, Z)&=g(\bar{\nabla}_{X}\varphi Y,\varphi Z)= g(\bar{\nabla}_{X}\mathcal{T}Y,\varphi Z)
 \end{align*}	
here we have used $\mathcal{F}Y=0$ for any $Y\in\Gamma(\mathcal{D}).$ Now, by using  \eqref{vx1} and \eqref{gauss}, we have
	\begin{align}
	g(\bar{\nabla}_{X}Y, Z)&= g({\nabla}_{X}\mathcal{T}Y +h(X,\mathcal{T}Y),\varphi \mathcal{Q}Z+\varphi \mathcal{R}Z)\nonumber\\
		&=g(\nabla_{X}\mathcal{T}Y+h(X,\mathcal{T}Y),\mathcal{T}\mathcal{Q}Z+\mathcal{F}\mathcal{Q}Z+\mathcal{T}\mathcal{R}Z +\mathcal{F}\mathcal{R}Z))\nonumber\\
		&=g(\nabla_{X}\mathcal{T}Y, \mathcal{T}\mathcal{Q}Z+\mathcal{T}\mathcal{R}Z) +g(h(X,\mathcal{T}Y), \mathcal{F}\mathcal{Q}Z+\mathcal{F}\mathcal{R}Z))\nonumber\\
		&=g(\nabla_{X}\mathcal{T}Y,\mathcal{T}Z)+g(h(X,\mathcal{T}Y),\mathcal{F}Z) \label{dpar1-}
	\end{align}
for any $X,Y\in\Gamma(\mathcal{D})$ and $Z=\mathcal{Q}Z+\mathcal{R}Z\in\Gamma(\mathcal{D}_1\oplus \mathcal{D}_2).$  Now, for any $X,Y\in\Gamma(\mathcal{D})$ and $W\in\Gamma(TM)^\perp,$ we have 
\begin{align}
	g(\bar{\nabla}_{X}Y,W)&=-g(\varphi\bar{\nabla}_{X_1}\varphi Y_1, W)=-g(\varphi(\nabla_{X}\mathcal{T}Y+h(X,\mathcal{T}Y)),W))\nonumber\\
	&=-g(\mathcal{T}\nabla_{X}\mathcal{T}Y+\mathcal{F}\nabla_{X}\mathcal{T}Y +\mathcal{B}h(X,\mathcal{T}Y)+\mathcal{C}h(X,\mathcal{T}Y),W))\nonumber\\
	&=-g(\mathcal{F}\nabla_{X}\mathcal{T}Y+\mathcal{C}h(X,\mathcal{T}Y),W) \label{dpar2-}
\end{align}
Thus proof follows \eqref{dpar}, \eqref{dpar1-} and \eqref{dpar2-}.
\end{proof}
\begin{theorem}\label{d1paralel}
	Let $M$ be a quasi bi-slant submanifolds of $\bar{M}.$ The slant distribution $\mathcal{D}_1$ defines totally geodesic foliation on $M$ if and only if 
		\begin{align}
	&g(\mathcal{A}_{\mathcal{F}\mathcal{T}V_1}U_1,Z)-g(\mathcal{A}_{\mathcal{F}V_1}U_1,\mathcal{T}\mathcal{P}Z)
 \nonumber 	\\
	&=g(\mathcal{A}_{\mathcal{F}V_1}U_1,\mathcal{T}\mathcal{R}Z)-g(\nabla^{\perp}_{U_1}\mathcal{F}V_1,\mathcal{F}\mathcal{R}Z) \label{d1par1}
	\end{align}
	and
	\begin{align}\label{d1par2}
	\mathcal{F}\mathcal{A}_{\mathcal{F}V_1}U_1-\nabla^{\perp}_{U_1}\mathcal{F}\mathcal{T}V_1-\mathcal{C}\nabla^{\perp}_{U_1}\mathcal{F}V_1\in\Gamma(TM)
	\end{align}
	for any $X, Y\in\Gamma(\mathcal{D}), \ Z=\mathcal{Q}Z+\mathcal{R}Z\in\Gamma(\mathcal{D}_1\oplus \mathcal{D}_2)$ and $W\in\Gamma(TM)^\perp.$
\end{theorem}	

\begin{proof}
	The distribution $\mathcal{D}_1$ defines a totaly geodesic foliation on $M$ if and only if $g(\bar{\nabla}_{U_1}V_1,\xi )=0$,  $g(\bar{\nabla}_{U_1}V_1,Z)=0$ and $g(\bar{\nabla}_{U_1}V_1,W )=0$, for  any  $U_1,V_1\in\Gamma(\mathcal{D}_1),$ $Z=\mathcal{P}Z+\mathcal{R}Z\in\Gamma(\mathcal{D}_1\oplus \mathcal{D}_2)$ and $W\in\Gamma(TM)^{\perp}.$ Since $M$ is a cosymplectic manifold, we immediately have $g(\bar{\nabla}_{U_1}V_1,\xi )=0.$ Now, for any $U_1,V_1\in\Gamma(\mathcal{D}_1),$ and $Z=\mathcal{P}Z+\mathcal{R}Z\in\Gamma(\mathcal{D}_1\oplus \mathcal{D}_2),$ by using 
	\eqref{gphixphiy} and \eqref{nablaxxi}, we obtain
\begin{align*}
	g(\bar{\nabla}_{U_1}V_1,Z)&=
	-g(\bar{\nabla}_{U_1}\varphi \mathcal{T}V_1,Z) + g(\bar{\nabla}_{U_1} \mathcal{F}V_1, \varphi \mathcal{P}Z+\varphi \mathcal{R}Z).
\end{align*}	
	Now, by using lemma \eqref{lemtkare} (i), we get
	\begin{align*}
	g(\bar{\nabla}_{U_1}V_1,Z)&=
	cos^2\theta_1g(\bar{\nabla}_{U_1} V_1, Z)-g(-\mathcal{A}_{\mathcal{F}\mathcal{T}V_1}U_1+\nabla^{\perp}_{U_1}\mathcal{F}\mathcal{T}V_1,Z)\\
	&+g(-\mathcal{A}_{\mathcal{F}V_1}U_1+\nabla^{\perp}_{U_1}\mathcal{F}V_1,\mathcal{T}\mathcal{P}Z )\\
	&+g(-\mathcal{A}_{\mathcal{F}V_1}U_1 +\nabla^{\perp}_{U_1}\mathcal{F}V_1,\mathcal{T}\mathcal{R}Z +\mathcal{F}\mathcal{R}Z)
	\end{align*}
	or
	\begin{align}
	sin^2\theta_1g(\bar{\nabla}_{U_1}V_1,Z)&=g(\mathcal{A}_{\mathcal{F}\mathcal{T}V_1}U_1,Z)-g(\mathcal{A}_
	{\mathcal{F}V_1}U_1,\mathcal{T}\mathcal{P}Z)\nonumber \\
	&-g(\mathcal{A}_{\mathcal{F}V_1}U_1,\mathcal{T}\mathcal{R}Z)+g(\nabla^{\perp}_{U_1}\mathcal{F}V_1,\mathcal{F}\mathcal{R}Z).\label{d1par} 
\end{align}
Now, for any $U_1,V_1\in\Gamma(\mathcal{D})$ and $W\in\Gamma(TM)^\perp,$ we have 
\begin{align*}
	g(\bar{\nabla}_{U_1}V_1,W)&
	=-g(\bar{\nabla}_{U_1}\varphi \mathcal{T}V_1,W)-g(\varphi(\bar{\nabla}_{U_1}\mathcal{F}V_1),W)\\ &-g(\varphi(-\mathcal{A}_{\mathcal{F}V_1}U_1+\nabla^{\perp}_{U_1}\mathcal{F}V_1),W)\\
	&=\cos^2\theta_1g(\bar{\nabla}_{U_1}V_1,W)-g(-\mathcal{A}_{\mathcal{F}\mathcal{T}V_1}U_1+\nabla^{\perp}_{U_1}\mathcal{F}\mathcal{T}V_1,W)\\
	&-g(-\mathcal{T}\mathcal{A}_{\mathcal{F}V_1}U_1-\mathcal{F}\mathcal{A}_{\mathcal{F}V_1}U_1+
	\mathcal{B}\nabla^{\perp}_{U_1}\mathcal{F}V_1+\mathcal{C}\nabla^{\perp}_{U_1}\mathcal{F}V_1,W)
\end{align*}
or
\begin{align}
\sin^2\theta_1g(\bar{\nabla}_{U_1}V_1,W)&=-g(\nabla^{\perp}_{U_1}\mathcal{F}\mathcal{T}V_1,W)
+g(\mathcal{F}\mathcal{A}_{\mathcal{F}V_1}U_1-\mathcal{C}\nabla^{\perp}_{U_1}\mathcal{F}V_1,W)\nonumber \\
&=g(\mathcal{F}\mathcal{A}_{\mathcal{F}V_1}U_1-\nabla^{\perp}_{U_1}\mathcal{F}\mathcal{T}V_1 -\mathcal{C}\nabla^{\perp}_{U_1}\mathcal{F}V_1,W)\label{d1par1-}
\end{align}
Thus proof follows \eqref{d1par} and \eqref{d1par1-}.
\end{proof}

\begin{theorem}\label{d2paralel}
	Let $M$ be a quasi bi-slant submanifolds of $\bar{M}.$ The slant distribution $\mathcal{D}_2$ defines totally geodesic foliation on $M$ if and only if 
	\begin{align}\label{d2par0}
	\mathcal{T}(\nabla_{U_2}\mathcal{T}V_2-\mathcal{A}_{\mathcal{F}V_2}U_2)\in\Gamma(\mathcal{D}_2),
	\end{align}
	\begin{align}\label{d2par1}
	\mathcal{B}(h(U_2,TV_2)+\nabla^{\perp}_{U_2}\mathcal{F}V_2)\in\Gamma(TM)^{\perp}
	\end{align}
	and
	\begin{align}\label{d2par2}
	g(\nabla^{\perp}_{U_2}\mathcal{F}\mathcal{T}V_2 -\mathcal{F}\mathcal{A}_{V_2}U_2,W)=g(\nabla^{\perp}_{U_2}\mathcal{F}V_2,\mathcal{C}W)
	\end{align}
	for any $U_2, V_2\in\Gamma(\mathcal{D}_2), \ Z=\mathcal{P}Z+\mathcal{Q}Z\in\Gamma(\mathcal{D}\oplus \mathcal{D}_1)$ and $W\in\Gamma(TM)^\perp.$
\end{theorem}	

 From theorem \eqref{dparalel}, \eqref{d1paralel} and \eqref{d2paralel}, we have the following decomposition theorem:
 \begin{theorem}
 	Let $M$ be a proper quasi bi-slant submanifolds of a cosmyplectic manifold $\bar{M}.$ Then $M$ is a local product Riemannian manifold of the form $M_{\mathcal{D}}\times M_{\mathcal{D}_1}\times M_{\mathcal{D}_2},$ where $M_{\mathcal{D}},$ $M_{\mathcal{D}_1}$ and $M_{\mathcal{D}_2}$ are leaves of $\mathcal{D},$ $\mathcal{D}_1$ and $\mathcal{D}_2,$ recpectively, if and only if the conditions \eqref{dpar1}, \eqref{dpar2}, \eqref{d1par1}, \eqref{d1par2}, \eqref{d2par0}, \eqref{d2par1} and \eqref{d2par2} hold.
 \end{theorem}

\section{Quasi bi-slant submanifolds with parallel canonical structures}

In this section, we obtain some results for the quasi bi-slant submanifolds with parallel canonical structure.
Let $M$ be a proper quasi bi-slant submanifold of a cosymplectic manifold $\bar{M}.$ Then we define

\begin{equation}\label{z1tz2}
(\bar{\nabla}_{Z_1}\mathcal{T})Z_2=\nabla_{Z_1}\mathcal{T}Z_2-\mathcal{T}\nabla_{Z_1}Z_2
\end{equation}
\begin{equation}\label{z1fz2}
(\bar{\nabla}_{Z_1}\mathcal{F})Z_2=\nabla^{\perp}_{Z_1}\mathcal{F}Z_2-\mathcal{F}\nabla_{Z_1}Z_2
\end{equation}
\begin{equation}\label{z1bz2}
(\bar{\nabla}_{Z_1}\mathcal{B})W_1=\nabla_{Z_1}\mathcal{B}W_1-\mathcal{B}\nabla^{\perp}_{Z_1}W_1
\end{equation}
\begin{equation}\label{z1cz2}
(\bar{\nabla}_{Z_1}\mathcal{C})W_1=\nabla^{\perp}_{Z_1}\mathcal{C}W_1-\mathcal{C}\nabla^{\perp}_{Z_1}W_1
\end{equation}
where $Z_1, Z_2 \in\Gamma(TM)$ and $W_1\in\Gamma(TM)^{\perp}$.\\
Then, the endomorphism $\mathcal{T}$ $(resp. ~~ \mathcal{F})$ and the endomorphism $\mathcal{B}$ $(resp. ~~ \mathcal{C})$ are parallel if $\bar{\nabla}\mathcal{T}\equiv 0$ $(resp.~~ \bar{\nabla}\mathcal{F} \equiv 0)$ and $\bar{\nabla}\mathcal{B}\equiv0$ $(resp.~~ \bar{\nabla}\mathcal{C}\equiv0),$ respectively.\\
Taking into account of \eqref{teget}, \eqref{normal}, \eqref{teget1}, \eqref{normal1} and \eqref{z1tz2}-\eqref{z1cz2}, we have the following lemma.

\begin{lemma}
	Let $M$ be a quasi bi-slant submanifold of a cosymplectic manifold $\bar{M}$. Then for any $Z_1, Z_2\in\Gamma(TM)$ and $W_1 \in \Gamma(TM)^{\perp}$ we obtain	
	\begin{equation}\label{z1tz21}
	(\bar{\nabla}_{Z_1}\mathcal{T})Z_2=\mathcal{A}_{\mathcal{F}Z_2}Z_1 +\mathcal{B}h(Z_1,Z_2)
	\end{equation}
	\begin{equation}\label{z1fz21}
	(\bar{\nabla}_{Z_1}\mathcal{F})Z_2=\mathcal{C}h(Z_1,Z_2)-h(Z_1,\mathcal{T}Z_2)
	\end{equation}
	\begin{equation}\label{z1bz21}
	(\bar{\nabla}_{Z_1}\mathcal{B})W_1=\mathcal{A}_{\mathcal{C}W_1}Z_1 +\mathcal{T}A_{W_1}Z_1
	\end{equation}
	\begin{equation}\label{z1cz21}
	(\bar{\nabla}_{Z_1}\mathcal{C})W_1=-\mathcal{F}\mathcal{A}_{W_1}Z_1 - h(Z_1,\mathcal{B}W_1).
	\end{equation}
\end{lemma}

First, we have the following theorem:
\begin{theorem}
	Let $M$ be a quasi bi-slant submanifold of a cosymplectic manifold $\bar{M}$. Then, $\mathcal{T}$ is parallel if and only if the invariant distribution $\mathcal{D}$ is totally geodesic.
\end{theorem}
\begin{proof}
	For any $X,Y\in\Gamma(\mathcal{D}),$ from \eqref{z1tz21}, we have 
	\begin{equation}\label{nxty}
	(\bar{\nabla}_X\mathcal{T})Y=\mathcal{B}h(X,Y)
	\end{equation} 
 here we have used $\mathcal{A}_{\mathcal{F}Y}X=0$ since $\mathcal{F}Y=0$ for any $Y\in\Gamma(\mathcal{D})$. Thus, our assertion comes from  \eqref{nxty}.
\end{proof}
\begin{theorem}
	Let $M$ be a quasi bi-slant submanifold of a cosymplectic manifold $\bar{M}$. Then if $\mathcal{F}$ is parallel if and only if
	\begin{equation}\label{acvz2z1}
	g(\mathcal{A}_{\mathcal{C}V}Z_2,Z_1)=-g(\mathcal{A}_VZ_1,\mathcal{T}Z_2)
	\end{equation}
	for any $Z_1, Z_2\in\Gamma(TM)$ and $V\in\Gamma(TM)^{\perp}.$
\end{theorem}
\begin{proof}
	Assume that $F$ is parallel. Now, from \eqref{z1fz21}, we have 
	\begin{equation}
	(\bar{\nabla}_{Z_1}\mathcal{F})Z_2=\mathcal{C}h(Z_1,Z_2)-h(Z_1,\mathcal{T}Z_2).
	\end{equation}
	Now, taking inner product with $V\in\Gamma(TM)^{\perp}$ in the above equation and using \eqref{gauss}, we obtain
	\begin{align*}
		g((\bar{\nabla}_{Z_1}\mathcal{F})Z_2,V)&=g(\mathcal{C}h(Z_1,Z_2)-h(Z_1,\mathcal{T}Z_2),V)\\
		&=g(\mathcal{C}h(Z_1,Z_2),V)-g(h(Z_1,\mathcal{T}Z_2),V)\\
		&=-g(h(Z_1,Z_2),\varphi V)-g(\bar{\nabla}_{Z_1}\mathcal{T}Z_2,V)\\
		&=-g(\mathcal{A}_{\mathcal{C}V}Z_2,Z_1)
		+g(\mathcal{T}Z_2,\bar{\nabla}_{Z_1}V)\\
		&=-g(\mathcal{A}_{\mathcal{C}V}Z_2,Z_1)+g(\mathcal{T}Z_2,-\mathcal{A}_VZ)
	\end{align*}
	which gives the assertion.
\end{proof}
\begin{theorem}
Let $M$ be a quasi bi-slant submanifold of a cosymplectic manifold $\bar{M}$. Then $\mathcal{F}$ is parallel if and only if $\mathcal{B}$ is parallel. 
\end{theorem}
\begin{proof}
	By using \eqref{gauss}, \eqref{z1fz21} and \eqref{z1bz21}, we get
	\begin{align*}
		g((\bar{\nabla}_{Z_1}\mathcal{F})Z_2,W_1)&=g(\mathcal{C}h(Z_1,Z_2),W_1)-g(h(Z_1,\mathcal{T}Z_2),W_1)\\
		&=-g(h(Z_1,Z_2),\mathcal{C}W_1)-g(\mathcal{A}_{W_1}Z_1,\mathcal{T}Z_2)\\
		&=-g(\mathcal{A}_{\mathcal{C}W_1}Z_1,Z_2)+g(\mathcal{T}\mathcal{A}_{W_1}Z_1,Z_2)\\
		&=-g(\mathcal{A}_{\mathcal{C}W_1}Z_1-\mathcal{T}\mathcal{A}_{W_1}Z_1,Z_2)\\
		&=-g((\bar{\nabla}_{Z_1}\mathcal{B})W_1,Z_2)
	\end{align*}
	for any $Z_1, Z_2\in\Gamma(TM)$ and $W_1\in\Gamma(TM)^{\perp}$. This proves our assertion.
\end{proof}

Finally, we mention another non-trivial example of quasi bi-slant submanifold of a cosymplectic manifold.
\begin{example}
	Let $M$ be a submanifold of $\mathbb{R}^{11}$ defined by
	$$x(u,v,t,r,s,k,z)=(u,v,t,\frac{1}{\sqrt{2}}r,\frac{1}{\sqrt{2}}r,0,s,k\cos\alpha ,k\sin\alpha ,0,z).$$
	We can easily to see that the tangent bundle of $M$ is spanned by the tangent
	vectors
	$$e_1=\frac{\partial}{\partial x_1}, e_2=\frac{\partial}{\partial y_1}, e_3=\frac{\partial}{\partial x_2}, e_4=\frac{1}{\sqrt{2}}\frac{\partial}{\partial y_2}+\frac{1}{\sqrt{2}}\frac{\partial}{\partial x_3},$$
	$$e_5=\frac{\partial}{\partial x_4}, e_6=\cos\alpha\frac{\partial}{\partial y_4}+\sin\alpha\frac{\partial}{\partial x_5}
	,e_7=\frac{\partial}{\partial z}=\xi.$$
We define the almost contact structurev $\varphi$ of $\mathbb{R}^{11},$ by
$$\varphi(\frac{\partial}{\partial x_i})=\frac{\partial}{\partial y_i},\ \  \varphi(\frac{\partial}{\partial y_j})=-\frac{\partial}{\partial x_j},\ \  \varphi(\frac{\partial}{\partial z})=0,\ \ 1\leq i,j\leq 5.$$
For any vector field $Z=\lambda_i \frac{\partial}{\partial x_i}+\mu_j \frac{\partial}{\partial y_j}+\nu\frac{\partial}{\partial z}\in\Gamma(T\mathbb{R}^{11}),$ then we have
$$g(Z,Z)=\lambda^{2}_i+\mu^{2}_j+\nu^{2}, \ \ g(\varphi Z, \varphi Z)=\lambda^{2}_i+\mu^{2}_j$$ and
$$\varphi^{2}Z=-\lambda_i\frac{\partial}{\partial x_i}-\mu_j\frac{\partial}{\partial y_j}=-Z$$ 
for any $i,j=1,...,5.$ It follows that $g(\varphi Z, \varphi Z)=g(Z,Z)-\eta^{2}(Z).$ Thus $(\varphi, \xi,\eta,g)$ is an
is an almost contact metric structure on $\mathbb{R}^{11}$. Thus we have

$$\varphi e_1=\frac{\partial}{\partial y_1}, \ \varphi e_2=\frac{\partial}{\partial x_1},\  \varphi e_3=\frac{\partial}{\partial y_2},\  \varphi e_4=-\frac{1}{\sqrt{2}}\frac{\partial}{\partial x_2}+\frac{1}{\sqrt{2}}\frac{\partial}{\partial y_3},$$
$$\varphi e_5=\frac{\partial}{\partial y_4}, \ \varphi e_6=-\cos\alpha\frac{\partial}{\partial x_4}+\sin\alpha\frac{\partial}{\partial y_5}, \ \varphi e_7=0.$$
By direct calculations, we obtain the distribution $\mathcal{D}=span\{e_1,e_2\}$ is an invariant distribution, the  distribution $\mathcal{D}_1=span\{e_3,e_4\}$ is a slant distribution with slant angle $\theta_1=\frac{\pi}{4}$ and the distribution $\mathcal{D}_2=span\{e_5,e_6\}$ is also a slant distribution with slant angle $\theta_2=\alpha, \ 0<\alpha<\frac{\pi}{2}.$ Thus $M$ is a $7-$dimensional proper quasi bi-slant submanifold of $\mathbb{R}^{11}$ with its usual almost contact metric structure.	
\end{example}

\end{document}